\documentclass{amsart}

\usepackage{amsmath, amsbsy,amsfonts,amsopn,amstext, euscript, amssymb,amsbsy,amsfonts,amsthm,latexsym,amsopn,amstext, amsxtra,euscript,amscd}

\input xy \xyoption{all}

\newtheorem{theorem}{Theorem}
\newtheorem{proposition}{Proposition}
\newtheorem{lemma}{Lemma}
\newtheorem{remark}{Remark}

\theoremstyle{definition}

\theoremstyle{remark}

\title{Genus two curves covering elliptic curves: a computational approach}

\author{T. Shaska}

\address{Department of Mathematics\\
Oakland University\\
Rochester, MI, 48309-4485.}

\email{shaska@oakland.edu }

\def\bP{\mathbb P}
\def\bC{\mathbb C}
\def\C{\mathcal C}
\def\Z{\mathbb Z}
\def\X{\mathcal X}

\def\H{\mathcal H}
\def\M{\mathcal M}
\def\L{\mathcal L}
\def\iso{{\, \cong\, }}
\def\lar{\longrightarrow}

\def\bZ{{\mathbb Z}}
\def\s{\sigma}

\def\X{\mathcal X}
\def\H{\mathcal H}
\def\Hs{\mathcal H_\s}
\def\A{\mathcal A}
\def\D{\Delta}

\def\u{\mathfrak u}

\def\({\left(}
\def\){\right)}
\def\cO{{\mathcal O}}
\def\a{\alpha}

\def\bG{\bold G}
\def\g{\gamma}
\def\e{\varepsilon}

\def\P{\mathcal P}
\def\r{\delta}  % ?????
\def\<{\langle}
\def\>{\rangle}

\def\sem{\rtimes }
\def\a{{\alpha}}
\def\th{{\theta}}

\def\sem{{\rtimes}}
\def\iso{{\, \cong\, }}
\def\e{{\xi}}
\def\r{{r}}

\def\s{\sigma}

\def\P{\mathcal P}
\def\<{\langle}
\def\>{\rangle}
\def\_u{{\mathfrak u}}
\def\J{J_{48}}
\def\U{\mathcal U}
\def\u{{u}}
\def\v{{v}}
\def\i{i}

\begin{document}

\maketitle

\begin{abstract}
A genus 2 curve $C$ has an elliptic subcover if there exists a degree $n$ maximal covering $\psi: C \to E$ to
an elliptic curve $E$. Degree $n$ elliptic subcovers occur in pairs  $(E, E')$. The Jacobian $J_C$ of $C$ is
isogenous of degree $n^2$ to the product $E \times E'$. We say that $J_C$ is $(n, n)$-split. The locus of
$C$, denoted by $\L_n$, is an algebraic subvariety of the moduli space $\M_2$. The space $\L_2$ was studied
in Shaska/V\"olklein \cite{SV1} and Gaudry/Schost \cite{GS}. The space $\L_3$ was studied in \cite{Sh1} were
an algebraic description was given as sublocus of $\M_2$.

In this survey  we give a brief description of the spaces $\L_n$ for  a general $n$ and then focus on small
$n$. We describe some of  the computational details which were skipped in \cite{SV1} and \cite{Sh1}. Further
we explicitly describe the relation between the elliptic subcovers $E$ and $E'$.  We have implemented most of
these relations in computer programs which check easily whether a genus 2 curve has  $(2, 2)$ or $(3, 3)$
split Jacobian. In each case the elliptic subcovers can be explicitly computed.
\end{abstract}

\section{Introduction}
%**************************************************************

Let $C$ be a genus 2 curve defined over an algebraically closed field $k$, of characteristic zero. Let $\psi: C
\to E$ be a degree $n$ maximal covering (i.e. does not factor through an isogeny) to an elliptic curve $E$ defined
over $k$. We say that $C$ has a \emph{degree n elliptic subcover}. Degree $n$ elliptic subcovers occur in pairs.
Let $(E, E')$ be such a pair. It is well known that there is an isogeny of degree $n^2$ between the Jacobian $J_C$
of $C$ and the product $E \times E'$. We say that $C$ has \textbf{(n,n)-split Jacobian}. The locus of such $C$,
denoted by $\L_n$, is a 2-dimensional algebraic subvariety of the moduli space $\M_2$ of genus two curves.

In this survey we study the genus 2 curves with $(n, n)$-split Jacobian for small $n$. While such curves have been
studied by many authors, our approach is simply computational. Some of the results have appeared in previous
articles of the author.

Curves of genus 2 with elliptic subcovers go back to Legendre and Jacobi. Legendre, in his \emph{Th\'eorie
des fonctions elliptiques}, gave the first example of a genus 2 curve with degree 2 elliptic subcovers. In a
review of Legendre's work, Jacobi (1832) gives a complete description for $n=2$. The case $n=3$ was studied
during the 19th century from Hermite, Goursat, Burkhardt, Brioschi, and Bolza. For a history and background
of the 19th century work see Krazer \cite[pg. 479]{Kr}. Cases when $n> 3$ are more difficult to handle. Frey
and Kani note the difficulty to get explicit examples, see Frey~\cite{Fr} and Frey, Kani \cite{FK}.

%*****************************************************************************

In \S 2 we give a brief description of genus 2 curves and their isomorphism classes which are classified by the
absolute invariants of binary sextics. Further, we display the list of groups that occur as full automorphism
groups of genus 2 curves defined over a field of characteristic $\neq 2$.

In \S 3 we study degree $n$ covers $\C \to E$ from a genus 2 curve to an elliptic curve. Such covers induce a
degree $n$ covering $\phi: \bP^1 \to \bP^1$. A careful study of such covers leads to determining an equation for
the curves $\C$. The covering $\phi: \bP^1 \to \bP^1$ could have different ramification structure. All such
structures are described in section 3.

The moduli space of coverings $\phi: \bP^1 \to \bP^1$ with fixed ramification structure is a Hurwitz space. The
irreducibility of such space, dimension, and the genus (in the case 1-dimensional spaces) can be computed via the
braid action.  For $n$ an odd integer we display such results in section 4. There is a natural morphism between
the Hurwitz space and the locus $\L_n$ (cf. \S 4). In the second part of section 4 we describe the correspondence
between the points of $\L_n$ and the Humbert space of discriminant $n^2$ which we denote by $\H_{n^2}$.

%**********************************************************8
In section 5 we study genus 2 curves with degree 2 elliptic subcovers. Jacobi \cite{J} gives a general form of
such curves: $Y^2=X^6-s_1X^4+s_2X^2-1$, and a description of $\L_2$ in terms of the cross ratios of the roots
$\a_1, \dots , \a_6$ of the sextic:
$$ \frac {\a_3-\a_1} {\a_3-\a_2} : \frac {\a_4-\a_1} {\a_4-\a_2}= \frac {\a_5-\a_1} {\a_5-\a_2} : \frac
{\a_6-\a_1} {\a_6-\a_2}
$$
Thus, $\L_2$ is parameterized by the pair $(s_1, s_2)\in k^2$. We note that this parametrization of $\L_2$ factors
through a ramified Galois covering: $k^2 \lar k^2$, $(s_1,s_2) \to (u,v)$, where $u=s_1 s_2$ and $v=s_1^3+s_2^3$.
This induces a birational parametrization of $\L_2$ by the pairs $(u,v)$. All our computations use these
coordinates $(u, v)$. We use this to compute an equation for $\L_2$ in terms of the classical invariants. We give
a general relation between the j-invariants of degree 2 elliptic subfields of $K$. This improves \cite{GS}, where
each isomorphism type of $G$ is treated separately. We determine conditions when degree 2 elliptic subfields of
$K$ are 2 or 3-isogenous. For a generalization of such invariants $u$, and $v$ see Remark~\ref{remark1} or
Gutierrez/Shaska \cite{GS}.

%****************************************************

In section 6, we study the case $n=3$. We show that every genus 2 curve with a degree 3 elliptic subcover can
be written in the form
$$Y^2=(X^3+aX^2+bX+1)(4X^3+b^2X^2+2bX+1 ) $$
for $a,b \in k$. So $\L_3$ is parameterized by the pairs $(a, b) \in k^2$. The invariants of the two cubics
$r_1, r_2$  give a birational parametrization of $\L_3$. This parametrization of $\L_3$ factors through
ramified Galois coverings of degree 3 (resp. 2)
\begin{equation}
\begin{split}
k^2 & \to k^2 \to k^2\\
(a,b) & \to (u,v) \to (r_1,r_2)
\end{split}
\end{equation}
where $ab=u$ and $b^3=v$. The equation of $\L_3$ is computed in terms of the absolute invariants and is
displayed in \cite[Appendix A]{Sh1}. If $\C\in \L_3$ then $Aut( \C) $ is isomorphic to $\bZ_2, V_4$, $D_4$ or
$D_6$. Moreover, there are exactly six genus 2 curves with automorphism group $ D_4$ or $D_6$. The rational
models of these 12 curves and rational points on them were studied in \cite{Sh6}. We determine the
j-invariants of the elliptic subcovers and show that they satisfy the Fricke polynomial of level 2.

In the last section we give information on computer programs that we have made available for such computations.

%************************************************************************************
\clearpage

\section{Preliminaries}
%***********************************************************************
Let $k$ be an algebraically closed field of  characteristic zero and $C$ a genus 2 curve defined over $k$. Then
$C$ can be described as a double cover of $\bP^1(k)$ ramified in 6 places $w_1, \dots , w_6$. This sets up a
bijection between isomorphism classes of genus 2 curves and unordered distinct 6-tuples $w_1, \dots , w_6 \in
\bP^1 (k)$ modulo automorphisms of $\bP^1 (k) $. An unordered 6-tuple $\{w_i\}_{i=1}^6$ can be described by a
binary sextic (i.e. a homogenous equation $f(X,Z)$ of degree 6). Let  $\M_2$ denote the moduli space of genus 2
curves. To describe $\M_2$ we need to find polynomial functions of the coefficients of a binary sextic $f(X,Z)$
invariant under linear substitutions in $X,Z$ of determinant one. These invariants were worked out by Clebsch and
Bolza in the case of zero characteristic  and generalized by Igusa for any characteristic different from 2; see
\cite{Bo}, \cite{Ig}, or \cite{SV1} for a more modern treatment.

Consider a binary sextic, i.e. a homogeneous polynomial $f(X,Z)$ in $k[X,Z]$ of degree 6:
$$f(X,Z)=a_6 X^6+ a_5 X^5Z+\dots  +a_0 Z^6.$$
\emph{Igusa  $J$-invariants} $\, \, \{ J_{2i} \}$ of $f(X,Z)$ are homogeneous polynomials of degree $2i$ in
$k[a_0, \dots , a_6]$, for $i=1,2,3,5$; see \cite{Ig}, \cite{SV1} for their definitions. Here $J_{10}$ is simply
the discriminant of $f(X,Z)$. It vanishes if and only if the binary sextic has a multiple linear factor. These
$J_{2i}$    are invariant under the natural action of $SL_2(k)$ on sextics. Dividing such an invariant by another
one of the same degree gives an invariant under $GL_2(k)$ action.

Two genus  2 fields $K$ (resp., curves) in the standard form $Y^2=f(X,1)$ are isomorphic if and only if the
corresponding sextics are $GL_2(k)$ conjugate. Thus if $I$ is a $GL_2(k)$ invariant (resp., homogeneous $SL_2(k)$
invariant), then the expression $I(K)$ (resp., the condition $I(K)=0$) is well defined. Thus the $GL_2(k)$
invariants are functions on the  moduli space $\mathcal M_2$ of genus 2 curves. This $\mathcal M_2$ is an affine
variety with coordinate ring
$$k[\mathcal M_2]=k[a_0, \dots , a_6, J_{10}^{-1}]^{GL_2(k)}$$
which is  the subring of degree 0 elements in $k[J_2, \dots ,J_{10}, J_{10}^{-1}]$. The \emph{ absolute
invariants}
$$ i_1:=144 \frac {J_4} {J_2^2}, \, \,  i_2:=- 1728 \frac {J_2J_4-3J_6} {J_2^3},\, \,  i_3 :=486 \frac {J_{10}}
{J_2^5}, $$
are even $GL_2(k)$-invariants. Two genus 2 curves with $J_2\neq 0$ are isomorphic if and only if they have
the same absolute invariants. If  $J_2=0 $ then we can define new invariants as in \cite{Sh2}. For the rest
of this paper if we say ``there is a genus 2 curve $\C$ defined over $k$'' we will mean the $k$-isomorphism
class of $\C$. We have the following; see \cite[Theorem 2]{SV1}.

\begin{lemma}\label{groups} The  automorphism group $G$ of a genus 2 curve $\C$ in characteristic $\ne2$ is
isomorphic  to \ $\bZ_2$, $\bZ_{10}$, $V_4$, $D_8$, $D_{12}$, $\bZ_3 \sem D_8$, $ GL_2(3)$, or $2^+S_5$. The case
when $G \iso 2^+S_5$ occurs only in characteristic 5. If $G \iso \bZ_3 \sem D_8$ (resp., $ GL_2(3)$) then $\C$ has
equation $Y^2=X^6-1$ (resp., $Y^2=X(X^4-1)$). If $G \iso \bZ_{10}$ then $\C$ has equation $Y^2=X^6-X$.
\end{lemma}
%

%&&&&&&&&&&&&&&&&&&&&&&&&&&&&&&&&&&&&&&&&&&&&&&&&&&&&&&&&&&&&&&&&&&
\section{Curves of genus 2 with split Jacobians}

Let $C$ and $E$ be curves of genus 2 and 1, respectively. Both are smooth, projective curves defined over $k$,
$char(k)=0$. Let $\psi: C \longrightarrow E$ be a covering of degree $n$. From the Riemann-Hurwitz formula,
$\sum_{P \in C}\, (e_{\psi}\,(P) -1)=2$ where $e_{\psi}(P)$ is the ramification index of points $P \in C$, under
$\psi$. Thus, we have two points of ramification index 2 or one point of ramification index 3. The two points of
ramification index 2 can be in the same fiber or in different fibers. Therefore, we have the following cases of
the covering $\psi$:\\

\textbf{Case I:} There are $P_1$, $P_2 \in C$, such that $e_{\psi}({P_1})=e_{\psi}({P_2})=2$, $\psi(P_1) \neq
\psi(P_2)$, and     $\forall  P \in C\setminus \{P_1,P_2\}$,  $e_{\psi}(P)=1$.

\textbf{Case II:} There are $P_1$, $P_2 \in C$, such that $e_{\psi}({P_1})=e_{\psi}({P_2})=2$, $\psi(P_1) =
\psi(P_2)$, and     $\forall  P \in C\setminus \{P_1,P_2\}$,  $e_{\psi}(P)=1$.

\textbf{Case III:} There is $P_1 \in C$ such that $e_{\psi}(P_1)=3$, and $ \forall P \in C \setminus \{P_1\}$,
$e_{\psi}(P)=1$.\\

\noindent In case I (resp. II, III) the cover $\psi$ has 2 (resp. 1) branch points in E.

Denote the hyperelliptic involution of $C$ by $w$. We choose $\mathcal O$ in E such that $w$ restricted to
$E$ is the hyperelliptic involution on $E$. We denote the restriction of $w$ on $E$ by $v$, $v(P)=-P$. Thus,
$\psi \circ w=v \circ \psi$. E[2] denotes the group of 2-torsion points of the elliptic curve E, which are
the points fixed by $v$. The proof of the following two lemmas is straightforward and will be omitted.

\begin{lemma} \label{lem_1}
a) If $Q \in E$, then $\forall P \in \psi^{-1}(Q)$, $w(P) \in \psi^{-1}(-Q)$.

b) For all $P\in C$, $e_\psi(P)=e_\psi\,({w(P)})$.
\end{lemma}

Let $W$ be the set of points in C fixed by $w$. Every curve of genus 2 is given, up to isomorphism, by a binary
sextic, so there are 6 points fixed by the hyperelliptic involution $w$, namely the Weierstrass points of $C$. The
following lemma determines the distribution of the Weierstrass points in fibers of 2-torsion points.

\begin{lemma}\label{lem2} The following hold:
\begin{enumerate}
\item $\psi(W)\subset E[2]$
\item If $n$ is an odd number then
 i) $\psi(W)=E[2]$
 ii) If $ Q \in E[2]$ then \#$(\psi^{-1}(Q) \cap W)=1 \mod (2)$
\item If $n$ is an even number then for all $Q\in E[2]$, \#$(\psi^{-1}(Q) \cap W)=0 \mod (2)$
\end{enumerate}
\end{lemma}

Let $\pi_C: C \lar \bP^1$ and $\pi_E:E \lar \bP^1$ be the natural degree 2 projections. The hyperelliptic
involution permutes the points in the fibers of $\pi_C$ and $\pi_E$. The ramified points of $\pi_C$, $\pi_E$
are respectively points in $W$ and $E[2]$ and their ramification index is 2. There is $\phi:\bP^1 \lar \bP^1$
such that the diagram commutes.
\begin{equation}
\begin{matrix}
C & \buildrel{\pi_C}\over\lar & \bP^1\\
\psi \downarrow &  & \downarrow \phi \\
E & \buildrel{\pi_E}\over\lar & \bP^1
\end{matrix}
\end{equation}
Next, we will determine the ramification of induced coverings $\phi:\bP^1 \lar \bP^1$. First we fix some
notation. For a given branch point we will denote the ramification of points in its fiber as follows. Any
point $P$ of ramification index $m$ is denoted by $(m)$. If there are $k$ such points then we write $(m)^k$.
We omit writing symbols for unramified points, in other words $(1)^k$ will not be written. Ramification data
between two branch points will be separated by commas. We denote by $\pi_E (E[2])=\{q_1, \dots , q_4\}$ and
$\pi_C(W)=\{w_1, \dots ,w_6\}$.

\subsubsection{The Case When $n$ is Odd}
%***************************************************************
The following theorem classifies the ramification types for the induced coverings $\phi:\bP^1 \lar \bP^1$
when the degree $n$ is odd.

\begin{theorem} \label{thm1}
Let $\psi:C \lar E$ be a covering of odd degree $n$ and $\phi: \bP^1 \lar \bP^1$ be the induced covering
induced by $\psi$. This induces a partitioning of the set of 6 Weierstrass points of $C$ into two sets
$W^{(1)}=W^{(1)}(C, E)$ and $W^{(2)}=W^{(2)}(\C, E)$, each of cardinality 3 such that $|\phi(W^{(1)})|=1$ and
$|\phi(W^{(2)})|=3$. Then the ramification structure of $\phi$ is as follows.
\begin{description}
\item[Case I] (the generic case)

 $\left ( (2)^\frac {n-1} 2 , (2)^\frac {n-1} 2 , (2)^\frac {n-1} 2 , (2)^ \frac {n-3} 2 ,(2)^1
\right ) $
\par Or  the following degenerate cases:

\item[Case II] (the 4-cycle case and the dihedral case)

 i) $ \left ((2)^\frac {n-1 }{ 2} ,(2)^\frac{n-1 }{ 2} , (2)^\frac {n-1 }{ 2} , (4)^1 (2)^\frac {n-7
}{ 2} \right )$

 ii) $\left ((2)^\frac {n-1}{ 2} , (2)^\frac {n-1}{2} ,
      (2)^\frac {n-1}{ 2}, (2)^\frac {n-1} 2 \right )$

 iii) $\left ((2)^\frac {n-1}{ 2} , (2)^\frac {n-1}{2} ,
      (4)^1  (2)^\frac {n-5}{ 2}, (2)^\frac {n-3} 2 \right )$

\item[Case III] (the 3-cycle case)

 i) $ \left ((2)^\frac {n-1 }{ 2} , (2)^\frac {n-1 }{ 2} , (2)^\frac {n-1}{ 2} , (3)^1 (2)^\frac {n-5
}{ 2} \right )$

 ii) $\left ( (2)^\frac {n-1} 2 , (2)^\frac {n-1} 2 , (3)^1 (2)^\frac {n-3} 2, (2)^\frac {n-3} 2
\right )$
\end{description}
\end{theorem}

\subsubsection{The Case When $n$ is Even}
%**************************************************************
Let us assume now that $deg(\psi)=n$ is an even number. The following theorem classifies the induced
coverings in this case.

\begin{theorem}\label{thm2}
If $n$ is an even number then the generic case for $\psi: C \lar E$ induce the following three cases for
$\phi: \bP^1 \lar \bP^1$:

\begin{description}
\item[I]    $ \left ( (2)^\frac {n-2} 2  , (2)^\frac {n-2} 2  , (2)^\frac {n-2} 2  , (2)^ \frac {n}
2 , (2) \right ) $
\item[II] $ \left ( (2)^\frac {n-4} 2 , (2)^\frac {n-2} 2 , (2)^\frac {n} 2 , (2)^ \frac {n} 2 ,
(2) \right ) $
\item[III] $ \left ( (2)^\frac {n-6} 2 , (2)^\frac {n} 2 , (2)^\frac {n} 2 , (2)^ \frac {n} 2 , (2)
\right ) $
\end{description}
Each of the above cases has the following degenerations (two of the branch points collapse to one)

\begin{description}
\item[I]
\begin{enumerate}
\item $\left ( (2)^\frac {n} 2  , (2)^\frac {n-2} 2  ,
 (2)^\frac {n-2} 2  , (2)^ \frac {n} 2        \right )
$ \item $\left ( (2)^\frac {n-2} 2 , (2)^\frac {n-2} 2 , (4) (2)^\frac {n-6} 2 , (2)^ \frac {n} 2 \right ) $
\item $\left ( (2)^\frac {n-2} 2 , (2)^\frac {n-2} 2 ,
 (2)^\frac {n-2} 2  , (4) (2)^ \frac {n-4} 2        \right )
$ \item $\left ( (3) (2)^\frac {n-4} 2 , (2)^\frac {n-2} 2 ,
 (2)^\frac {n-2} 2  ,  (2)^ \frac {n} 2        \right )
$
\end{enumerate}
\item[II]
\begin{enumerate}
\item   $\left ( (2)^\frac {n-2} 2  , (2)^\frac {n-2} 2  , (2)^\frac {n} 2  ,  (2)^ \frac {n} 2
\right ) $
\item $\left ( (2)^\frac {n-4} 2 , (2)^\frac {n} 2 , (2)^\frac {n} 2 , (2)^ \frac {n} 2 \right)$
\item $\left ((4) (2)^\frac {n-8} 2, (2)^\frac {n-2} 2 , (2)^\frac {n} 2, (2)^ \frac {n}2\right )$
\item $\left ( (2)^\frac {n-4} 2 , (4) (2)^\frac {n-6} 2 , (2)^\frac {n} 2 , (2)^ \frac {n} 2
\right ) $ \item $\left ( (2)^\frac {n-4} 2 , (2)^\frac {n-2} 2 , (2)^\frac {n-4} 2 , (2)^ \frac {n} 2 \right
) $ \item $\left ((3) (2)^\frac {n-6} 2 , (2)^\frac {n-2} 2 , (4) (2)^\frac {n} 2 , (2)^ \frac {n} 2 \right )
$ \item $\left ( (2)^\frac {n-4} 2 , (3) (2)^\frac {n-4} 2 ,
 (2)^\frac {n} 2  ,  (2)^ \frac {n} 2        \right )
$
\end{enumerate}
\item[III]
\begin{enumerate}
\item    $\left ( (2)^\frac {n-4} 2  , (2)^\frac {n} 2  ,
 (2)^\frac {n} 2  , (4) (2)^ \frac {n} 2        \right )
$ \item $\left ( (2)^\frac {n-6} 2 , (4) (2)^\frac {n-4} 2 ,
 (2)^\frac {n} 2  ,  (2)^ \frac {n} 2        \right )
$ \item $\left ( (2)^\frac {n} 2 , (2)^\frac {n} 2 ,
 (2)^\frac {n} 2  , (4) (2)^ \frac {n-10} 2        \right )
$ \item $\left ( (3) (2)^\frac {n-8} 2 , (2)^\frac {n} 2 ,
 (2)^\frac {n} 2  ,  (2)^ \frac {n} 2        \right )
$
\end{enumerate}
\end{description}
\end{theorem}

\subsection{Maximal coverings $\psi:C \lar E$.}
%****************************************************

Let $\psi_1:C \lar E_1$ be a covering of degree $n$ from a curve of genus 2 to an elliptic curve. The
covering $\psi_1:C \lar E_1$ is called a \textbf{maximal covering} if it does not factor through a nontrivial
isogeny. A map of algebraic curves $f: X \to Y$ induces maps between their Jacobians $f^*: J_Y \to J_X$ and
$f_*: J_X \to J_Y$. When $f$ is maximal then $f^*$ is injective and $ker (f_*)$ is connected, see \cite{Sh7}
for details.

Let $\psi_1:C \lar E_1$ be a covering as above which is maximal. Then ${\psi^*}_1: E_1 \to J_C$ is injective
and the kernel of $\psi_{1,*}: J_C \to E_1$ is an elliptic curve which we denote by $E_2$; see \cite{FK} or
\cite{Ku}. For a fixed Weierstrass point $P \in C$, we can embed $C$ to its Jacobian via
\begin{equation}
\begin{split}
i_P: C & \lar J_C \\
 x & \to [(x)-(P)]
\end{split}
\end{equation}
Let $g: E_2 \to J_C$ be the natural embedding of $E_2$ in $J_C$, then there exists $g_*: J_C \to E_2$. Define
$\psi_2=g_*\circ i_P: C \to E_2$. So we have the following exact sequence
$$ 0 \to E_2 \buildrel{g}\over\lar J_C \buildrel{\psi_{1,*}}\over\lar E_1 \to 0 $$
The dual sequence is also exact
$$ 0 \to E_1 \buildrel{\psi_1^*}\over\lar J_C \buildrel{g_*}\over\lar E_2 \to 0 $$
If $deg (\psi_1)$ is an odd number then the maximal covering $\psi_2: C \to E_2$ is unique (up to isomorphism
of elliptic curves), see Kuhn \cite{Ku}. If the cover $\psi_1:C \lar E_1$ is given, and therefore $\phi_1$,
we want to determine $\psi_2:C \lar E_2$ and $\phi_2$. The study of the relation between the ramification
structures of $\phi_1$ and $\phi_2$ provides information in this direction. The following lemma (see
\cite[pg. 160]{FK}) answers this question for the set of Weierstrass points $W=\{P_1, \dots , P_6\}$ of C
when the degree of the cover is odd.

\begin{lemma} Let $\psi_1:C \lar E_1$, be maximal of  degree $n$.
Then, the map $\psi_2: C\to E_2$ is a maximal covering of degree $n$. Moreover,
\begin{enumerate}
\item [i) ] if  $n$ is  odd and ${\cO}_i\in E_i[2]$, $i=1, 2$  are  the places such that $\#
(\psi_i^{-1}({\cO }_i)\cap W) = 3$, then $\psi_1^{-1}({\cO }_1)\cap W$ and $\psi_2^{-1}({\cO }_2)\cap W$ form
a disjoint union of $W$.
\item [ii)] if $n$ is even and $Q\in E[2]$, then $\# \left( \psi^{-1}(Q)\right) = 0$ or 2.
\end{enumerate}
\end{lemma}
The above lemma says that if $\psi$ is maximal of even degree then the corresponding induced covering can
have only type \textbf{I} ramification, see theorem \ref{thm2}.

\section{The locus of genus two curves with $(n, n)$ split Jacobians}
%****************************************************************

In this section we will discuss the Hurwitz spaces of coverings with ramification as in the previous section and
the Humbert spaces of discriminant $n^2$.

\subsection{Hurwitz spaces of covers $\phi : \bP^1 \to \bP^1$}
%********************************************************************

Two covers $f:X\to \bP^1$ and $f':X'\to \bP^1$ are called \textbf{weakly equivalent} if there is a
homeomorphism $h:X\to X'$ and an analytic automorphism $g$ of $\bP^1$ (i.e., a Moebius transformation) such
that  $g\circ f=f'\circ h$. The covers $f$ and $f^\prime$ are called \textbf{equivalent } if the above holds
with $g=1$.

Consider a cover $f:X \to \bP^1$ of degree $n$, with branch points $p_1,...,p_r\in \bP^1$. Pick $p\in \bP^1
\setminus\{p_1,...,p_r\}$, and choose loops $\gamma_i$ around $p_i$ such that $\gamma_1,...,\gamma_r$ is a
standard generating system of the fundamental group $\Gamma:=\pi_1( \bP^1 \setminus\{p_1,...,p_r\},p)$, in
particular, we have $\gamma_1 \cdots \gamma_r=1$. Such a system $\gamma_1,...,\gamma_r$ is called a homotopy basis
of $\bP^1 \setminus\{p_1,...,p_r\}$. The group $\Gamma$ acts on the fiber $f^{-1}(p)$ by path lifting, inducing a
transitive subgroup $G$ of the symmetric group $S_n$ (determined by $f$ up to conjugacy in $S_n$). It is called
the \textbf{monodromy group} of $f$. The images of $\gamma_1,...,\gamma_r$ in $S_n$ form a tuple of permutations
$\s=(\s_1,...,\s_r)$ called a tuple of \textbf{branch cycles} of $f$.

We say a cover $f:X\to\bP^1$ of degree $n$ is of type $\s$ if it has $\s$ as tuple of branch cycles relative to
some homotopy basis of $\bP^1$ minus the branch points of $f$. Let $\Hs$ be the set of weak equivalence classes of
covers of type $\s$.  The \textbf{Hurwitz space} $\Hs$ carries a natural structure of an quasiprojective variety.

We have $\Hs=\H_\tau$ if and only if the tuples $\s$, $\tau$ are in the same \textbf{braid orbit} $\mathcal
O_\tau = \mathcal O_\sigma$. In the case of the covers $\phi : \bP^1 \to \bP^1$ from above, the corresponding
braid orbit consists of all tuples in $S_n$ whose cycle type matches the ramification structure of $\phi$.

This and   the genus of $\Hs$ in the degenerate cases  (see the following table) has been computed in GAP  by the
BRAID PACKAGE written by K. Magaard.

\begin{table}[!ht]
\begin{tabular}{c|c|c|c|c|c|c }
deg& Case &   cycle type of $\s$ & $\# ( \mathcal O_\s )$     & $G$ & $\dim \Hs$ & genus of $\Hs$ \\
\hline & &      && &   &   \\
3 & & $(2^2, 2^2, 2^2, 2, 2) $  &  40   & $S_3 $   & 2 & -- \\
 & 1 & $(2^2, 2^2, 4, 2) $    &  8   &$S_5 $   & 1  & 0 \\
 & 2 & $(2^2, 2^2, 2\cdot 3, 2)$& 6  & $S_5$    & 1  & 0 \\
 & 3 & $(2^2, 2^2, 2^2, 3)$  &  9  &  $A_5$    & 1  & 1 \\
 & 4     &                       &     &           &    &   \\
  & &      & &    & & \\
5 & & $(2^2, 2^2, 2^2, 2, 2) $  &  40   & $S_5 $   & 2 & -- \\
 & 1 & $(2^2, 2^2, 4, 2) $    &  8   &$S_5 $   & 1  & 0 \\
 & 2 & $(2^2, 2^2, 2\cdot 3, 2)$& 6  & $S_5$    & 1  & 0 \\
 & 3 & $(2^2, 2^2, 2^2, 3)$  &  9  &  $A_5$    & 1  & 1 \\
 & 4 &      & &    & & \\
  &  &      & &    & & \\
7 & & $(2^2, 2^2, 2^2, 2, 2) $  &  168   & $S_7 $   & 2 & -- \\
\end{tabular}
%\bigskip
%
%\caption{Table }
\end{table}

\subsection{Humbert surfaces}
%**************************************************************************************************
Let $\A_2$ denote the moduli space of principally polarized abelian surfaces. It is well known that $\A_2$ is
the quotient of the Siegel upper half space $\mathfrak H_2$ of symmetric complex $2 \times 2$ matrices with
positive definite  imaginary part by the action of the symplectic group $Sp_4 (\Z)$; see \cite[p. 211]{G}.

Let $\D$ be a fixed positive integer and  $N_\D$ be the set of matrices
$$\tau =
\begin{pmatrix}z_1 & z_2\\
z_2 & z_3
\end{pmatrix}
\in \mathfrak H_2$$ such that there exist nonzero integers $a, b, c, d, e $ with the following properties:
\begin{equation}\label{humb}
\begin{split}
& a z_1 + bz_2 + c z_3 + d( z_2^2 - z_1 z_3) + e = 0\\
& \D= b^2 - 4ac - 4de\\
\end{split}
\end{equation}

The \emph{  Humbert surface}  $\H_\D$  of discriminant $\D$   is called the image of $N_\D$ under the
canonical map
$$\mathfrak H_2 \to \A_2:= Sp_4( \Z)\setminus{\mathfrak H}_2,$$
see \cite{Hu, BW, Mu} for details.  It is known that $\H_\D \neq \emptyset$ if and only if $\D
> 0$ and $\Delta \equiv 0 \textit { or } 1 \mod 4$. Humbert (1900) studied the zero loci in
Eq.~\eqref{humb} and discovered certain relations between points in these spaces and certain plane
configurations of six lines; see \cite{Hu} for more details.

For a genus 2 curve $C$ defined over $\bC$, $[C]$ belongs too $\L_n$ if and only if the isomorphism class $[J_C]
\in \A_2$ of its (principally polarized) Jacobian $J_C$ belongs to the Humbert surface $\H_{n^2}$, viewed as a
subset of the moduli space $\A_2$ of principally polarized abelian surfaces; see \cite[Theorem 1, p. 125]{Mu}  for
the proof of this statement. In \cite{Mu} is shown that there is a one to one correspondence between the points in
$\L_n$ and points in $\H_{n^2}$. Thus, we have the map:
\begin{equation}
\begin{split}
&  \H_\s \, \,  \longrightarrow \, \, \L_n  \, \, \longrightarrow  \, \, \H_{n^2}\\
([f], (p_1, & \dots , p_r)  \to [\X] \to [J_\X]\\
\end{split}
\end{equation}
In particular, every point in $\H_{n^2}$ can be represented by an element of $\mathfrak H_2$ of the form
$$\tau =
\begin{pmatrix}z_1 & \frac 1 n \\
\frac 1 n & z_2
\end{pmatrix}, \qquad z_1, \, z_2 \in \mathfrak H.
$$
There have been many attempts to explicitly describe these Humbert surfaces. For some small discriminant this
has been done by several authors; see \cite{SV1}, \cite{Sh1}, \cite{Ku}. Geometric characterizations of such
spaces for $\D= 4, 8, 9$, and 12 were given by Humbert (1900) in \cite{Hu} and for $\D= 13, 16, 17, 20$, 21
by Birkenhake/Wilhelm (2003) in \cite{BW}.

\section{Genus 2 curves with  degree 2 elliptic subcovers}
%*************************************************************************

An \textbf{elliptic involution} of $K$ is  an  involution in $G$ which is different from $z_0$ (the
hyperelliptic involution). Thus the elliptic involutions of $G$ are in 1-1 correspondence with the elliptic
subfields of $K$ of degree 2 (by the Riemann-Hurwitz formula).

If $z_1$ is an elliptic involution and $z_0$ the hyperelliptic one, then $z_2:=z_0\, z_1$ is another elliptic
involution. So the elliptic involutions come naturally in pairs. This pairs also the elliptic subfields of
$K$ of degree 2. Two such subfields $E_1$ and $E_2$ are paired if and only if $E_1\cap k(X)=E_2\cap k(X)$.
$E_1$ and $E_2$ are $G$-conjugate unless $G\iso D_6$ or $G\iso V_4$ (This can be checked from
Lemma~\ref{groups}).

\begin{theorem} \label{mainthm_kap3}
Let $K$ be a genus 2 field and $e_2(K)$ the number of $Aut(K)$-classes of elliptic subfields of $K$ of degree
2. Suppose $e_2(K) \geq 1$. Then the classical invariants of $K$ satisfy the equation,
\begin{scriptsize}
\begin{equation}
\begin{split}\label{eq_L2_J}
-J_2^7 J_4^4+8748J_{10}J_2^4J_6^2507384000J_{10}^2J_4^2J_2-19245600J_{10}^2J_4J_2^3
-592272J_{10}J_4^4J_2^2 \\
-81J_2^3J_6^4-3499200J_{10}J_2J_6^3+4743360J_{10}J_4^3J_2J_6-870912J_{10}J_4^2J_2^3J_6\\
+1332J_2^4J_4^4J_6-125971200000J_{10}^3 +384J_4^6J_6+41472J_{10}J_4^5+159J_4^6J_2^3\\
-47952J_2J_4J_6^4+104976000J_{10}^2J_2^2J_6-1728J_4^5J_2^2J_6+6048J_4^4J_2J_6^2+108J_2^4J_4J_6^3\\
+12J_2^6J_4^3J_6+29376J_2^2J_4^2J_6^3-8910J_2^3J_4^3J_6^2-2099520000J_{10}^2J_4J_6-236196J_{10}^2J_2^5\\
+31104J_6^5-6912J_4^3J_6^34 +972J_{10}J_2^6J_4^2 +77436J_{10}J_4^3J_2^4-78J_2^5J_4^5\\
+3090960J_{10}J_4J_2^2J_6^2-5832J_{10}J_2^5J_4J_6-80J_4^7J_2-54J_2^5J_4^2J_6^2-9331200J_{10}J_4^2J_6^2 = 0&\\
\end{split}
\end{equation}
\end{scriptsize}

\noindent Further, $e_2(K)=2$ unless $K=k(X,Y)$ with $$ Y^2=X^5-X  $$ in which case $e_2(K)=1$.
\end{theorem}

\begin{proof} Since $e_2(K)$ is the number of conjugacy classes of elliptic involutions in $G$ the claim about $e_2(K)$
follows from theorem \ref{lem1}. For the proof of the following lemma see \cite{SV1}.

\begin{lemma} \label{lem1}
Suppose $z_1$ is an elliptic involution of $K$. Let $z_2=z_1z_0$, where $z_0$ is the hyperelliptic
involution. Let $E_i$ be the fixed field of $z_i$ for $i=1,2$. Then $K=k(X,Y)$ where
\begin{equation}
Y^2=X^6-s_1X^4+s_2X^2-1
\end{equation}
and $27-18s_1s_2-s_1^2s_2^2+4s_1^3+4s_2^3\neq 0$. Further $E_1$ and $E_2$ are the subfields $k(X^2,Y)$ and
$k(X^2, YX)$.
\end{lemma}

We need to determine to what extent the normalization in the above proof determines the coordinate $X$. The
condition $z_1(X)=-X$ determines the coordinate $X$ up to a coordinate change by some $\g\in \Gamma$
centralizing $z_1$. Such $\g$ satisfies $\g(X)=mX$ or $\g(X) = \frac m X$, $m\in k\setminus \{0\}$. The
additional condition $abc=1$ forces $1=-\g(\a_1) \dots \g(a_6)$, hence $m^6=1$. So $X$ is determined up to a
coordinate change by the subgroup $H\iso D_6$ of $\Gamma$ generated by $\tau_1: X\to \e_6X$, $\tau_2: X\to
\frac 1 X$, where $\e_6$ is a primitive 6-th root of unity. Let $\e_3:=\e_6^2$. The coordinate change by
$\tau_1$ replaces $s_1$ by $\e_3s_2$ and $s_2$ by $\e_3^2 s_2$. The coordinate change by $\tau_2$ switches
$s_1$ and $s_2$. Invariants of this $H$-action are:
\begin{equation} u:=s_1 s_2, \quad v:=s_1^3+s_2^3 \end{equation}

\begin{remark}\label{remark1}
Such invariants were quite important in simplifying  computations for the locus $\L_2$. Later they have been used
by Duursma and Kiyavash to  show that genus 2 curves with extra involutions are suitable for the vector
decomposition problem; see \cite{Du} for details. In this volume they are used again, see the paper by Cardona and
Quer. They were later generalized to higher genus hyperelliptic curves and were called dihedral invariants; see
\cite{GSh}.
\end{remark}

Classical invariants of the field $K$ given by lemma \ref{lem1} are:
\begin{equation}
\begin{split} \label{J_eq}
J_2 & =  240+16u\\
J_4 & =  48v +4u^2+1620-504u \\
J_6 & =  -20664u +96v -424u^2+24u^3+160uv+119880\\
J_{10} & = 64(27-18u-u^2+4v)^2
\end{split}
\end{equation}
For $J_2\neq 0$ we express the absolute invariants $i_1, i_2, i_3$ in terms of $u$ and $v$. We can eliminate
$u$ and $v$ and get the following equation of $\L_2$.
\begin{scriptsize}
\begin{equation}
\begin{split}
-27i_1^6+9i_1^7+161243136i_3i_1^3-12441600i_3i_2^3+2i_2^5+107495424i_3i_1^2i_2+54i_1^3i_2^2\\
-52254720i_3i_1i_2^2-47278080i_3i_1^3i_2-8294400i_3i_1^2i_2^2-9459597312000i_3^2i_1^2-18i_1^4i_2^2\\
-240734712102912i_3^2+111451255603200i_3^2i_1+20639121408000i_3^2i_2-55240704i_3i_1^4\\
+2i_1^6i_2-4i_1^3i_2^3+331776i_3i_1^5-27i_2^4-2866544640000i_3^2i_1i_2+161243136i_3i_2^2+9i_1i_2^4\\
-264180754022400000i_3^3  = 0 & \\
\end{split}
\end{equation}
\end{scriptsize}

To get rid of the condition $J_2\neq 0$ we multiply by $J_2^5$ to get the ``projective'' equation \eqref{eq_L2_J}
of $\L_2$. This holds indeed for all $K\in \L_2$, as can be checked by substituting from \eqref{J_eq}. This
completes the proof of Theorem~\ref{mainthm_kap3}. \\

\end{proof}

The following proposition determines the group $G$ in terms of $u$ and $v$.

%&&&&&&&&&&&&&&&&&&&&&&&&&&&&&&&&&&&&&&&&&
\begin{proposition} \label{u_v}
Let $\C$ be a genus 2 curve such that $G:=Aut(\C)$ has an elliptic involution and $J_2\neq 0$. Then,

a) $G\iso \bZ_3 \sem D_4$ if and only if $(u,v)=(0,0)$ or $(u,v)=(225, 6750)$.

b) $G\iso W_1$ if and only if $u=25$ and $v=-250$.

c) $G\iso D_6$ if and only if $4v-u^2+110u-1125=0$, for $u\neq 9, 70 + 30\sqrt{5}, 25$.

Moreover, the classical invariants satisfy the equations,
\begin{scriptsize}
\begin{equation}
\begin{split}
-J_4J_2^4+12J_2^3J_6-52J_4^2J_2^2+80J_4^3+960J_2J_4J_6-3600J_6^2 &=0\\
864J_{10}J_2^5+3456000J_{10}J_4^2J_2-43200J_{10}J_4J_2^3-2332800000J_{10}^2-J_4^2J_2^6\\
-768J_4^4J_2^2+48J_4^3J_2^4+4096J_4^5 &=0\\
\end{split}
\end{equation}
\end{scriptsize}

d) $G\iso D_4$ if and only if $v^2-4u^3=0$, for $u \neq 1,9, 0, 25, 225$. Cases $u=0,225$ and $u=25$ are
reduced to cases a),and b) respectively. Moreover, the classical invariants satisfy \eqref{eq_L2_J} and the
following equation,
\begin{small}
\begin{equation}
\begin{split}
1706J_4^2J_2^2+2560J_4^3+27J_4J_2^4-81J_2^3J_6-14880J_2J_4J_ 6+28800J_6^2 &=0
\end{split}
\end{equation}
\end{small}
\end{proposition}

\begin{proposition}\label{thm_3}  The mapping
$$A: (u,v) \lar (i_1, i_2, i_3)$$
gives a birational parametrization of $\L_2$. The fibers of A of cardinality $>1$ correspond to those curves $\C$
with $|Aut(\C)| > 4$.
\end{proposition}

\begin{proof} See \cite{SV1} for the details.
\end{proof}

\subsection{Elliptic subcovers}
%&&&&&&&&&&&&&&&&&&&&&&&&&&&&&&&&&&&&&&&&&&&&&&&&&&&&&&&&&&&&

Let $j_1$ and $j_2$ denote the j-invariants of the elliptic curves $E_1$ and $E_2$ from lemma \ref{lem1}. The
invariants $j_1$ and $j_2$ and are roots of the quadratic
\begin{equation}
\begin{split}
j^2+256\frac {(2u^3-54u^2+9uv-v^2+27v)} {(u^2+18u-4v-27)} j + 65536 \frac {(u^2+9u-3v)} {(u^2+18u-4v-27)^2} &
=0 \label{j_eq}
\end{split}
\end{equation}

\subsubsection{Isomorphic elliptic subcovers}
%&&&&&&&&&&&&&&&&&&&&&&&&&&&&&&&&&&&&&&&&&&&&&&&&&&&&&&&&&&&&&&&&&

The elliptic curves $E_1$ and $E_2$ are isomorphic when equation \eqref{j_eq} has a double root. The
discriminant of the quadratic is zero for
$$(v^2-4u^3)(v-9u+27)=0$$

\begin{remark} From lemma \ref{lem1},  $v^2=4u^3$ if and only if  $Aut(\C)\iso
D_4$. So for $\C$ such that $Aut(\C)\iso D_4$, $E_1$ is isomorphic to $E_2$. It is easily checked that $z_1$
and $z_2=z_0 z_1$ are conjugate when $G\iso D_4$. So they fix isomorphic subfields.
\end{remark}
\par If $v=9(u-3)$ then the  locus of these curves is  given by,
\begin{equation}
\begin{split}
4i_1^5-9i_1^4+73728i_1^2i_3 -150994944 i_3^2=0 \\
289i_1^3 - 729 i_1^2 +54 i_1i_2 -i_2^2=0 \\
\end{split}
\end{equation}
For $(u,v)=(\frac 9 4, - \frac {27} 4)$ the curve has $Aut(\C)\iso D_4$ and for $(u,v)=(137, 1206)$ it has
$Aut(\C)\iso D_6$. All other curves with $v=9(u-3)$
 belong to the general case, so $Aut(\C)\iso V_4$. The j-invariants of elliptic
curves are $j_1=j_2=256(9-u)$. Thus, these genus 2 curves are parameterized by the j-invariant of the
elliptic subcover.

\begin{remark}
This embeds the moduli space $\mathcal M_1$ into $\mathcal M_2$ in a functorial way.
\end{remark}

\subsection{Isogenous degree 2 elliptic subfields}
%&&&&&&&&&&&&&&&&&&&&&&&&&&&&&&&&&&&&&&&&&&&&&&&&&&&&&&&

In this section we study pairs of degree 2 elliptic subfields of $K$ which are 2 or 3-isogenous. We denote by
$\Phi_n(x,y)$ the n-th modular polynomial (see Blake et al. \cite{Blake} for the formal definitions. Two
elliptic curves with j-invariants $j_1$ and $j_2$ are $n$-isogenous if and only if $\Phi_n(j_1,j_2)=0$.

\subsubsection{3-Isogeny.} Suppose $E_1$ and $E_2$ are 3-isogenous. Then, from equation (\ref{j_eq})
and $\Phi_3(j_1,j_2)=0$ we eliminate $j_1$ and $j_2$. Then,
\begin{equation}
(4v-u^2+110u-1125)\cdot g_1(u,v)\cdot g_2(u,v)=0 \label{iso_3}
\end{equation}
where $g_1$ and $g_2$ are
\begin{small}
\begin{equation}
\begin{split}
g_1 & =-27008u^6+256u^7-2432u^5v+v^4+7296u^3v^2-6692v^3u-1755067500u\\
 & +2419308v^3-34553439u^4+127753092vu^2+16274844vu^3-1720730u^2v^2\\
 & -1941120u^5+381631500v+1018668150u^2-116158860u^3+52621974v^2\\
 & +387712u^4v -483963660vu-33416676v^2u+922640625 \\
\end{split}
\end{equation}
\end{small}
\begin{small}
\begin{equation}
\begin{split}
g_2 & =291350448u^6-v^4u^2-998848u^6v-3456u^7v+4749840u^4v^2+17032u^5v^2\\
 &  +4v^5+80368u^8+256u^9+6848224u^7-10535040v^3u^2-35872v^3u^3+26478v^4u\\
 &  -77908736u^5v+9516699v^4+307234984u^3v^2-419583744v^3u-826436736v^3\\
 &  +27502903296u^4+28808773632vu^2-23429955456vu^3+5455334016u^2v^2\\
 &  -41278242816v+82556485632u^2-108737593344u^3-12123095040v^2\\
 & +41278242816vu+3503554560v^2u+5341019904u^5-2454612480u^4v\\
\end{split}
\end{equation}
\end{small}
Thus, there is a isogeny of degree 3 between $E_1$ and $E_2$ if and only if $u$ and $v$ satisfy equation
\eqref{iso_3}. The vanishing of the first factor is equivalent to $G\iso D_6$. So, if $Aut(\C)\iso D_6$ then
$E_1$ and $E_2$ are isogenous of degree 3. This was also noted by Gaudry and Schost \cite{GS}.

\subsubsection {2-Isogeny}
%******************************************************
Below we give the modular 2-polynomial.
\begin{equation}
\begin{split}
\Phi_2 & =x^3-x^2y^2+y^3+1488xy(x+y)+40773375xy-162000(x^2-y^2)+\\
&8748000000(x+y)-157464000000000 \\
\end{split}
\end{equation}
Suppose $E_1$ and $E_2$ are isogenous of degree 2. Substituting $j_1$ and $j_2$ in $\Phi_2$ we get
\begin{equation}
 f_1(u,v)\cdot f_2(u,v)=0
\label{phi_2}
\end{equation}
where $f_1$ and $f_2$ are
\begin{small}
\begin{equation}
\begin{split}
f_1 &
=-16v^3-81216v^2-892296v-2460375+3312uv^2+707616vu+3805380u+\\
&  18360vu^2 -1296162u^2 -1744u^3v-140076u^3+801u^4+256u^5 \\
\end{split}
\end{equation}
\end{small}
\begin{small}
\begin{equation}
\begin{split}
f_2 & =4096u^7+256016u^6-45824u^5v+4736016u^5-2126736vu^4+23158143u^4\\
&  -25451712u^3v-119745540u^3+5291136v^2u^2-48166488vu^2-2390500350u^2\\
& -179712uv^3+35831808uv^2+1113270480vu+9300217500u-4036608v^3\\
& -1791153000v-8303765625-1024v^4+163840u^3v^2-122250384v^2+256u^2v^3 \\
\end{split}
\end{equation}
\end{small}

\subsubsection{Other isogenies between  elliptic subcovers}
%&&&&&&&&&&&&&&&&&&&&&&&&&&&&&&&&&&&&&&&&&&&&&&&&&&&&&&&&&&&&&&&

If $G\iso D_4$, then $z_1$ and $z_2$ are in the same conjugacy class. There are again two conjugacy classes
of elliptic involutions in $G$. Thus, there are two degree 2 elliptic subfields (up to isomorphism) of $K$.
One of them is determined by double root $j$ of the equation \eqref{j_eq}, for $v^2-4u^3=0$. Next, we
determine the j-invariant $j^\prime$ of the other degree 2 elliptic subfield and see how it is related to
$j$.
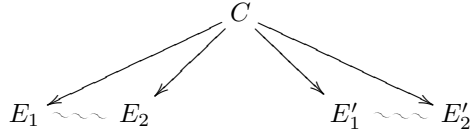
$$
%\begin{center}
\xymatrix{
 & & C  \ar@{->}[dl]  \ar@{->}[dll]   \ar@{->}[dr]  \ar@{->}[drr]    & & & \\
E_1  \ar@{~}[r]& E_2 & & E_1^\prime \ar@{~}[r]& E_2^\prime & \\
%\end{center}
}
$$
If $v^2-4u^3=0$ then $\bG\iso V_4$ and $\P=\{\pm 1, \pm \sqrt{a}, \pm \sqrt{b}\}$. Then, $s_1= a + \frac 1 a
+ 1=s_2$. Involutions of $\C$ are $\tau_1: X\to -X$, $\tau_2: X\to \frac 1 X$, $\tau_3: X\to - \frac 1 X$.
Since $\tau_1 $ and $\tau_3$ fix no points of $\P$ the they lift to involutions in $G$. They each determine a
pair of isomorphic elliptic subfields. The j-invariant of elliptic subfield fixed by $\tau_1$ is the double
root of equation (\ref{j_eq}), namely
$$ j= -256 \frac {v^3} {v+1} $$
To find the j-invariant of the elliptic subfields fixed by $\tau_3$ we look at the degree 2 covering $\phi:
\bP^1\to \bP^1$, such that $\phi(\pm 1)=0$, $\phi(a)=\phi(-\frac 1 a)=1$, $\phi(-a)=\phi(\frac 1 a)=-1$, and
$\phi(0)=\phi(\infty)=\infty$. This covering is, $\phi (X)= \frac {\sqrt{a}} {a-1} \frac {X^2-1} {X}$. The
branch points of $\phi$ are $q_i= \pm \frac {2i \sqrt{a}} {\sqrt{a-1}}$. From lemma \ref{lem1} the elliptic
subfields $E_1^\prime$ and $E_2^\prime$ have 2-torsion points $\{ 0, 1, -1,q_i\}$. The j-invariants of
$E_1^\prime$ and $E_2^\prime$ are
$$j^\prime= -16 \frac {(v-15)^3} {(v+1)^2}$$
Then $\Phi_2(j, j^\prime)=0$, so $E_1$ and $E_1^\prime$ are isogenous of degree 2. Thus, $\tau_1$ and
$\tau_3$ determine degree 2 elliptic subfields which are 2-isogenous.

%*************************************
%*************************************

\section{Genus 2 curves with degree 3 elliptic subcovers}
%&&&&&&&&&&&&&&&&&&&&&&&&&&&&&&&&&&&&&&&&&&&&&&&&&&&&&&&&&&&&&&&&

This case was studied in detail in\cite{Sh1}. The main theorem was:

\begin{theorem}\label{main_thm_deg3}
Let $K$ be a genus 2 field and $e_3(K)$ the number of $Aut(K/k)$-classes
 of elliptic subfields of $K$ of degree 3.  Then;

i)  $e_3(K) =0, 1, 2$, or  $4$

ii)    $e_3(K) \geq 1$ if and only if
 the classical invariants of $K$ satisfy  the irreducible
equation   $F(J_2, J_4, J_6, J_{10})=0$ displayed  in  \cite[Appendix A]{Sh1}.
\end{theorem}

There are exactly two genus 2 curves (up to isomorphism) with $e_3(K)=4$. The case $e_3(K)=1$ (resp., 2) occurs
for a 1-dimensional (resp., 2-dimensional)  family of genus 2 curves, see \cite{Sh1}.

%****************************************************************
\begin{lemma}\label{lemma3_1}
Let $K$ be a genus 2 field and $E$ an elliptic subfield of degree 3.

i) Then $K=k(X,Y)$ such that
\begin{equation}\label{n_3_form1}
Y^2= (4X^3+b^2X^2+2bX+1) (X^3+aX^2+b X+1)
\end{equation}
for $a, b \in k$ such that
\begin{equation}\label{n_3_form2}
\begin{split}
(4a^3+27-18ab-a^2b^2+4b^3)(b^3-27) \neq 0
\end{split}
\end{equation}
The roots of the first (resp. second) cubic correspond to $W^{(1)}(K,E)$, (resp. $W^{(2)}(K,E)$) in the
coordinates $X,Y$, (see theorem \ref{thm1}).

ii) $E=k(U,V)$ where $$U= \frac {X^2} {X^3+aX^2+b X+1}$$ and
\begin{equation}\label{eq_E1}
V^2= U^3 +2 \frac {ab^2-6a^2+9b} {R} U^2 + \frac {12a-b^2 } {R} U - \frac 4 {R}
\end{equation}
where $R=4a^3+27-18ab-a^2b^2+4b^3 \neq 0$.

iii) Define
$$ u:=ab, \quad v:=b^3 $$
Let $K^\prime$ be a genus 2 field and $E^\prime \subset K^\prime$ a degree 3 elliptic subfield. Let
$a^\prime, b^\prime$ be the associated parameters as above and $u^\prime:=a^\prime b^\prime$,
$v=(b^\prime)^3$. Then, there is a $k$-isomorphism $K \to K^\prime$ mapping $E\to E^\prime$ if and only if
exists a third root of unity $\e \in k$ with $a^\prime=\e a$ and $b^\prime=\e^2 b$. If $b\neq 0$ then such
$\e$ exists if and only if $v=v^\prime$ and $u=u^\prime$.

iv) The classical invariants of $K$ satisfy equation \cite[Appendix A]{Sh1}.
\end{lemma}

\medskip

Let
\begin{equation}
\begin{split}
F(X) & :=X^3+a X^2+b X+1\\
G(X) & :=4X^3+b^2X^2+2bX+1
\end{split}
\end{equation}

Denote by $R=4a^3+27-18ab-a^2b^2+4b^3$ the resultant of $F$ and $G$. Then we have the following lemma.
%&&&&&&&&&&&&&&&&&&&&&&&&&&&&&&&&&&&&&&&&&&&&&&&&&&&&&&&&&&&&&&&&&&&

\begin{lemma}\label{lemma3_2} Let $a,b \in k$ satisfy equation \eqref{n_3_form2}.
Then equation \eqref{n_3_form1} defines a genus 2 field $K=k(X,Y)$. It has elliptic subfields of degree 3,
$E_i=k(U_i, V_i)$, $i=1,2$, where $U_i$, and $V_i$ are as follows:
$$
U_1 = \frac {X^2} {F(X)}, \quad V_1= Y\, \frac {X^3-bX-2} {F(X)^2}
$$

\begin{small}
\begin{equation}
\begin{split}
U_2 & = \left\{ \aligned
  \frac {(X-s)^2 (X-t)} {G(X) } &\quad if \quad b(b^3-4ba+9) \neq 0 \\
  \frac { (3X- a)} {3(4X^3+1)} & \quad if \quad b= 0\\
  \frac { (bX+3)^2} {b^2G(X)}  &\quad if \quad (b^3-4ba+9) = 0\\
\endaligned
\right.
\end{split}
\end{equation}
\end{small}

where
$$ s=- \frac 3 b, \quad t=\frac {3a-b^2} {b^3-4ab+9}$$

\begin{scriptsize}
\begin{equation}
\begin{split}
V_2 & = \left\{ \aligned
 \frac {\sqrt {27-b^3} Y} {G(X)^2}
((4ab-8-b^3)X^3 -(b^2-4ab)X^2 +bX+1) & \quad if \quad b(b^3-4ba+9) \neq 0\\
 Y \frac {8X^3-4aX^2-1} {(4X^3+1)^2}   & \quad if \quad b= 0\\
\frac 8 b \sqrt{b} \frac Y {G(X)} (bX^3+9X^2+b^2X+b )    & \quad if \quad (b^3-4ba+9) = 0\\
\endaligned
\right.
\end{split}
\end{equation}
\end{scriptsize}
\end{lemma}

\begin{proof} We skip the details of the proof.

\end{proof}

%%%%%%%%%%%%%%%%%%%%%%%%%%%%%%%%%%%%%%%%%%%%%%%%%%%%%%

\subsection{Function field of $\L_3$}
%&&&&&&&&&&&&&&&&&&&&&&&&&&&&&&&&&&&&&&&&&&&&&&&&&&&&&&&&&&&&&&&&&&
The absolute invariants $i_1, i_2$, and $i_3$ are expressed in terms of $u,v$.  Let $\u, \v$ be independent
transcendentals over $k$ and $\i_1, \i_2, \i_3 \in k(\u,\v)$. Further elements $\r_1, \r_2 \in k(\u, \v)$ are
defined below; see \S~\ref{r1r2}.

From the resultants of equations if $i_1, i_2, i_3$ in terms of $u, v$,  we determine that
$[k(\v):k(\i_1,\i_2)]=16$, $[k(\v):k(\i_2,\i_3)]=40$, and $[k(\v):k(\i_1,\i_3)]=26$. We also can show that $\u\in
k(\i_1, \i_2, \i_3, \v)$, the expression is large and we display it on \cite[Appendix A]{Sh1}. Thus,
$[k(\u,\v):k(\i_1,\i_2, \i_3)] \leq 2$, see figure \ref{fig_2}.

\begin{figure}[!ht]
$$
\xymatrix{
& & k(\u,\v) \ar@{-}[d]_2       &  &\\
& & k(\i_1, \i_2, \i_3)=k(\r_1, \r_2) \ar@{-}[d]^{20} \ar@{-}[dr]^{13} \ar@{-}[dl]_8 & &\\
& k(\i_1, \i_2) & k(\i_1, \i_3) & k(\i_2, \i_3) }
$$
\caption{} \label{fig_2}
\end{figure}

Computing the equation \cite[Appendix A]{Sh1} directly from the equations of $i_1, i_2, i_3$ in terms of $u, v$,
exceeds available computer power. We use additional invariants $\r_1, \r_2$ to overcome this problem.
%&&&&&&&&&&&&&&&&&&&&&&&&&&&&&&&&&&&&&&&&&&&&&&&&&
\subsubsection{Invariants of Two Cubics}\label{r1r2}
%&&&&&&&&&&&&&&&&&&&&&&&&&&&&&&&&&&&&&&&&&&&&&&&&&&&&&&&&&&&&&&&&&&&&&&&&&&&
We define the following invariants of two cubic polynomials. For $F(X)=a_3X^3+a_2 X^2 + a_1 X+a_0$ and
$G(X)=b_3 X^3+b_2X^2+b_1X+b_0$ define
$$H(F,G) :=a_3 b_0 - \frac  1 3 a_2 b_1 + \frac 1 3 a_1 b_2 -a_0 b_3$$
We denote by $R(F,G)$ the resultant of $F$ and $G$ and by $D(F)$ the discriminant of $F$. Also,
$$r_1(F,G) =\frac  {H(F,G)^3} {R(F,G)}, \quad r_2(F,G)=\frac
{H(F,G)^4} {D(F)\, D(G)}$$

\begin{remark}
Note that $D(FG)= D(F) \cdot D(G) \cdot R^2(F, G)$.
\end{remark}

\noindent For
$$F(X)=X^3+aX^2+bX+1, \quad G(X)=4X^3 + b^2X^2 +2bX+1$$
from lemma \ref{lemma3_1} we have
\begin{equation}
\begin{split}\label{eq_r}
r_1(F,G) &= 27\frac {v(v-9-2u)^3} {4v^2-18uv+27v-u^2v+4u^3}\\
r_2 (F,G)& = -1296 \frac  {v(v-9-2u)^4}   {(v-27)(4v^2-18uv+27v-u^2v+4u^3)}\\
\end{split}
\end{equation}

\begin{remark}
Note that $r_1, r_2$ are defined for any $u,v$ by \eqref{n_3_form2}.
\end{remark}

Taking the resultants from the above equations we get the following equations for $\u$ and $\v$ over $k(\r_1,
\r_2)$:

\begin{small}
\begin{equation}
\begin{split} \label{eq_u}
65536\r_1\r_2^3 \,\u^2+(42467328\r_2^4+21233664\r_2^4\r_1+480\r_2\r_1^4+2\r_1^5+41472\r_2^2\r_1^3\\
+1548288\r_2^3\r_1^2-294912\r_2^3\r_1)\u-382205952\r_2^4+238878720\r_2^4\r_1-2654208\r_2^3\r_1\\
+13934592\r_2^3\r_1^2+285696\r_2^2\r_1^3+ 2400\r_2\r_1^4+7\r_1^5 & =0\\
\end{split}
\end{equation}
\end{small}
\begin{small}
\begin{equation}
\begin{split} \label{eq_v}
16384\v^2\r_2^3+(221184\r_2^3\r_1+\r_1^4+11520\r_2^2\r_1^2-442368\r_2^3+192\r_2\r_1^3)\v\\
-5971968\r_2^3\r_1-864\r_2\r_1^3-124416\r_2^2\r_1^2-2\r_1^4 & =0 \\
\end{split}
\end{equation}
\end{small}

%&&&&&&&&&&&&&&&&&&&&&&&&&&&&&&&&&&&&&&&&&&&&&&&&&&&&&&&&&&&&&&&&&&&&&&&&&&&&
In equation \eqref{eq_u} express $\r_1$ and $\r_2$ in terms of $\u$ and $\v$. Roots of this equation are $\u$
and $\nu(\u)$ where,

\begin{small}
\begin{equation}
\begin{split}
\nu(\u) & =\frac {(\v-3\u)(324\u^2+15\u^2\v-378\u\v-4\u\v^2+243\v+72\v^2)}
{(\v-27)(4\u^3+27\v-18\u\v-\u^2\v+4\v^2)}\\
\end{split}
\end{equation}
\end{small}
Similarly for $\v$ we get
\begin{small}
\begin{equation}
\begin{split}
\nu(\v) & =- \frac {4(\v-3\u)^3}{4\u^3+27\v-18\u\v-\u^2\v+4\v^2} \\
\end{split}
\end{equation}
\end{small}

Define a ring homomorphism
$$\nu: k[\u,\v] \to k(\u,\v)$$
$$\u \to \nu(\u), \quad \v \to \nu(\v)$$
Then, we compute $\nu^2=1$. Thus, $\nu$ extends to an involutory automorphism of $k(\u,\v)$ which we again
denote by $\nu$. Since,
$$\tau: k(\u,\v) \to k(\u, \v)$$
$$( \u, \v)  \to \left( \u, \nu(\v) \right)$$
is not involutory, then $[k(\u,\v):k(\r_1,\r_2)]=2$ and $Gal_{k(\u,\v)/k(\r_1, \r_2)}=\<\nu\>$.
%&&&&&&&&&&&&&&&&&&&&&&&&&&&&&&&&&&&&&&&&&&&&&&&&&&&&&&&&&&&&&&&&&&

\begin{lemma} The fields   $k(\i_1, \i_2, \i_3)=k(\r_1, \r_2)$ are the same.
\end{lemma}

\begin{remark}
To find the equation in \cite[Appendix A]{Sh1} we eliminate $\r_1$ and $\r_2$ from the three equations of the
above lemma. This equation has degree 8, 13, and 20 in $\i_1, \i_2, \i_3$ respectively.
\end{remark}

\begin{proof} \textbf{(Theorem~\ref{main_thm_deg3})}
%&&&&&&&&&&&&&&&&&&&&&&&&&&&&&&&&&&&&&&&&&&&&&&&&&&&&&&&&&&
The map
$$\th: (u,v) \to (i_1, i_2, i_3)$$  generically has degree 2, by
previous section. Denote the minors of the Jacobian matrix of $\th$ by $M_1(u,v), M_2(u,v), M_3(u,v)$. The
system

\begin{equation}
\begin{split}
\left\{ \aligned
  M_1(u,v)= 0 \\
  M_2(u,v)=0\\
  M_3(u,v)=0\\
\endaligned
\right.
\end{split}
\end{equation}
has solutions
\begin{equation}
\begin{split}\label{n_3_iso1}
8v^3+27v^2-54uv^2-u^2v^2+108u^2v+4u^3v-108u^3=0\\
\end{split}
\end{equation}
and 7 further solutions which we display in the following table  together with the corresponding values $(i_1,
i_2, i_3)$ and properties of the corresponding genus 2 field $K$.

\begin{figure}[ht!]
\renewcommand\arraystretch{1.5}
\noindent
$$
\begin{array}{|c|c|c|c|c|}
\hline
(u,v) & (i_1, i_2, i_3) & Aut(K) & e_3(K) \\
\hline
(-\frac 7 2, 2) & J_{10}=0, \quad \text{no associated} &  &  \\
 & \text{genus 2 field K}  &  & \\
\hline
(-\frac {775} 8, \frac {125} {96}), & & & \\
 (\frac {25} 2, \frac {250} {9})&
- \frac {8019} {20}, -\frac {1240029} {200}, \frac {531441} {100000}     & D_4 & 2 \\
\hline
(27- \frac {77} 2 \sqrt{-1}, 23+ \frac {77} 9 \sqrt{-1}), &  & & \\
(27+ \frac {77} 2 \sqrt{-1}, 23- \frac {77} 9 \sqrt{-1}) &
(\frac {729} {2116}, \frac {1240029}  {97336}, \frac {531441} {13181630464}  & D_4 & 2   \\
\hline
 (-15+ \frac {35} 8 \sqrt{5}, \frac {25} 2 + \frac {35} 6 \sqrt{5}), &  & & \\

(-15- \frac {35} 8 \sqrt{5}, \frac {25} 2 - \frac {35} 6 \sqrt{5})&
81, - \frac {5103} {25}, -\frac {729} {12500}  &  D_6 & 2  \\
\hline
\end{array}
$$
\caption{Corresponding $(u,v)$ for which the Jacobian matrix of $\theta$ is 0} \label{table1}
\end{figure}

Assume that equation \eqref{n_3_iso1} holds for some $(u,v)\in k^2$. Then the corresponding quantities $J_{2i}$,
$i=1,2,3, 5$  satisfy the equation

\begin{small}
\begin{equation}\label{eq_S}
F(J_2, J_4, J_6, J_{10})=0
\end{equation}
\end{small}
where $F(J_2, J_4, J_6, J_{10}) $ is displayed in \cite{Sh1}. This is obtained by taking the resultants of
equations of $i_1, i_2, i_3$  and \eqref{n_3_iso1}. We define $J_{48}:=F(J_2, J_4, J_6, J_{10}) $. By previous
section $\th$ is generically a covering of degree 2. So exists a Zariski open subset $\U$ of $k^2$ with the
following properties: Firstly, $\th$ is defined everywhere on $\U$ and is a covering of degree 2 from $\U$ to
$\th(\U)$. Further, if $\_u \in \U$ then all $\_u^\prime \in k^2$ with $\th$ defined at $\_u^\prime$ and
$\th(\_u^\prime)=\th(\_u)$ also lie in $\U$. Suppose $\underline i \in k^3$ such that $| \th^{-1}(\underline i)|
>2$ and $det(Jac(\th))$ does not vanish at any point of $\th^{-1}(\underline i) $.
Then by implicit function theorem, there is an open ball $B$ around each element of $\th^{-1}(\underline i) $
such that each point in $\th(B)$ has $> 2$ inverse images under $\th$. But $B$ has to intersect the Zariski
open set $\U$. This is a contradiction. Thus, if ${\underline i} \in k^3$ and $|\th^{-1}(\underline i)| > 2$,
then $det (Jac(\th)) =0$ at some point of $\th^{-1}(\underline i) $ and so $\J$ vanishes.

Let $e_3(K) >1$ and $J_2\neq 0$, $\J \neq 0$. Then $i_1, i_2, i_3$ are defined and by previous paragraph
$|\th^{-1}(i_1, i_2, i_3)| \leq 2$. Thus, by lemma \ref{lemma3_1} part iii) $e_3(K) \leq 2$. This completes the
proof of theorem ~\ref{main_thm_deg3}.

\end{proof}

%&&&&&&&&&&&&&&&&&&&&&&&&&&&&&&&&&&&&&&&&&&&&&&&&&&&&&&&&&&&&&&&&
\subsection{Elliptic subcovers}
%&&&&&&&&&&&&&&&&&&&&&&&&&&&&&&&&&&&&&&&&&&&&&&&&&&&&&&&

We express the j-invariants $j_i$ of the elliptic subfields $E_i$ of $K$, from lemma \ref{lemma3_2}, in terms
of $u$ and $v$ as follows:

\begin{small}
\begin{equation}\label{j_1}
\begin{split}
j_1 &= 16v \frac {(vu^2+216u^2-126vu-972u+12v^2+405v)^3}
{ (v-27)^3(4v^2+27v+4u^3-18vu-vu^2)^2}\\
j_2 &= -256 \frac {(u^2-3v)^3}{v(4v^2+27v+4u^3-18vu-vu^2)} \\
\end{split}
\end{equation}
\end{small}
where $v\neq 0, 27$.
\begin{remark}
The automorphism $\nu \in Gal_{k(u,v)/k(\r_1, \r_2)}$ permutes the elliptic subfields. One can easily check
that:
$$\nu(j_1)=j_2, \quad \nu(j_2)=j_1$$
\end{remark}
\noindent Define $T $ and $N $ as follows;
\begin{scriptsize}
\begin{equation}
\begin{split}
T &= \frac 1 {16777216r_2^3r_1^8}(1712282664960r_2^3r_1^6+1528823808r_2^4r_1^6+49941577728r_2^4r_1^5\\
& -38928384r_2^5r_1^5-258048r_2^6r_1^4+12386304r_2^6r_1^3+901736973729792r_2r_1^{10}\\
& +966131712r_2^5r_1^4+16231265527136256r_1^{10}+480r_2^8r_1+101376r_2^7r_1^2+479047767293952r_2r_1^8\\
& +7247757312r_2^3r_1^8+7827577896960r_2^2r_1^9+2705210921189376r_1^9+619683250176r_2^3r_1^7\\
& +21641687369515008r_1^{12}+32462531054272512r_1^{11}+r_2^9+37572373905408r_2^2r_1^7  \\
& +1408964021452800r_2r_1^9+45595641249792r_2^2r_1^8)\\
N &= - \frac 1 { 68719476736r_1^12r_2^3}(84934656r_1^5+1179648r_1^4r_2-5308416r_1^4-442368r_1^3r_2\\
& -13824r_1^2r_2^2-192r_1r_2^3-r_2^4)^3 \\
\end{split}
\end{equation}
\end{scriptsize}

\begin{lemma}
The j-invariants of the elliptic subfields satisfy the following quadratic equations over $k(r_1, r_2)$;
\begin{equation}\label{eq_j_new}
j^2- T \, j+ N =0
\end{equation}
\end{lemma}

\proof Substitute $j_1$ and $j_2$ as in \eqref{j_1} in equation \eqref{eq_j_new}. \endproof

\subsubsection{Isomorphic Elliptic Subfields}
%&&&&&&&&&&&&&&&&&&&&&&&&&&&&&&&&&&&&&&&&&&&&&&&&&&&&&&&&&&&&&&&&&&&&&&&&&&&&
Suppose that $E_1\iso E_2$. Then, $j_1=j_2$ implies that
\begin{small}
\begin{equation}
\begin{split}
8v^3+27v^2-54uv^2-u^2v^2+108u^2v+4u^3v-108u^3=0\\
\end{split}
\end{equation}
\end{small}
or
\begin{scriptsize}
\begin{equation}
\begin{split}\label{n_3_iso2}
& 324v^4u^2-5832v^4u+37908v^4-314928v^3u-81v^3u^4+255879v^3+30618v^3u^2\\
& -864v^3u^3-6377292uv^2 +8503056v^2-324u^5v^2+2125764u^2v^2-215784u^3v^2\\
& +14580u^4v^2+16u^6v^2+78732u^3v+8748u^5v -864u^6v-157464u^4v+11664u^6 =0\\
\end{split}
\end{equation}
\end{scriptsize}

The former equation is the condition that $det ( Jac(\th))=0$ see \eqref{eq_S}. From equation \ref{eq_S} and
expressions of $i_1, i_2, i_3$  we can express $u$ as a rational function in $i_1, i_2$, and $v$. This is
displayed in \cite[Appendix B]{Sh1}. Also, $[k(v):k(i_1)]=8$ and $[k(v):k(i_2)]=12$. Eliminating $v$ we get a
curve in $i_1$ and $i_2$ which has degree 8 and 12 respectively. Thus, $k(u,v)=k(i_1,i_2)$. Hence, $e_3(K) = 1$
for any $K$ such that the associated $u$ and $v$ satisfy equation \eqref{eq_S}.

\subsubsection{The Degenerate Case}
%&&&&&&&&&&&&&&&&&&&&&&&&&&&&&&&&&&&&&&&&&&&&&&&&&&&&&&&&
We assume now that one of the extensions $K/E_i$ from lemma \ref{lemma3_2} is degenerate, i.e. has only one
branch point. The following lemma determines a relation between $j_1$ and $j_2$.

\begin{lemma} Suppose that $K/E_2$ has only one branch point. Then,
$$729 j_1 j_2 -(j_2-432)^3=0$$
\end{lemma}

Making the substitution $T=-27j_1$ we get $$j_1=F_2(T)= \frac {(T+16)^3} T$$ where $F_2(T)$ is the Fricke
polynomial of level 2.

If both $K/E_1$ and $K/E_2$ are degenerate then
\begin{equation}
\begin{split}
\left\{ \aligned
729j_1j_2-(j_1-432)^3=0\\
729j_1j_2-(j_2-432)^3=0\\
\endaligned
\right.
\end{split}
\end{equation}
There are 7 solutions to the above system. Three of which give isomorphic elliptic curves
$$j_1=j_2=1728, \quad j_1=j_2=\frac 1 2 (297 \pm 81 \sqrt{-15})$$
The other 4 solutions are given by:
\begin{equation}
\begin{split}
\left\{ \aligned
729j_1j_2-(j_1-432)^3=0\\
j_1^2+j_2^2-1296(j_1+j_2)+j_1 j_2 +559872=0\\
\endaligned
\right.
\end{split}
\end{equation}
This corrects \cite{Ku} where it is claimed there is only one solution $j_1=j_2=1728$.

\section{Further remarks}
%&&&&&&&&&&&&&&&&&&&&&&&&&&&&&&&&&&&&&&&&&&&&&&&&&&&&&&&&&

If $e_3(\C)\geq 1$ then the automorphism group of $\C$ is one of the following: $\bZ_2, V_4$, $ D_4$, or $D_6$.
Moreover; there are exactly 6 curves $\C\in \L_3$ with automorphism group $D_4$ and six curves $\C\in \L_3$ with
automorphism group $D_6$. They are listed in \cite{Sh6} where rational points of such curves are found.

Genus 2 curves with  degree 5 elliptic subcovers are studied in \cite{SV2} where a description of the space $\L_5$
is given and all its degenerate loci. The case of degree 7 is the first case when all possible degenerate loci
occur.

We have organized the results of this paper in a Maple package which determines if a genus 2 curve has degree
$n=2,3$ elliptic subcovers. Further, all its elliptic subcovers are determined explicitly. We intend to implement
the results for $n=5$ and the degenerate cases for $n=7$.

%******************************************************************************************

\end{document}